\documentclass{amsart}

\usepackage{nicefrac}
\usepackage{hyperref}
\newcommand*{\mailto}[1]{\href{mailto:#1}{\nolinkurl{#1}}}

\newcommand{\doi}[1]{\href{http://dx.doi.org/#1}{DOI: #1}}

\newtheorem{theorem}{Theorem}[section]

\newtheorem{lemma}[theorem]{Lemma}
\newtheorem{corollary}[theorem]{Corollary}
\newtheorem{remark}[theorem]{Remark}
\newtheorem{hypothesis}[theorem]{Hypothesis}

\newcommand{\R}{{\mathbb R}}
\newcommand{\N}{{\mathbb N}}
\newcommand{\Z}{{\mathbb Z}}
\newcommand{\C}{{\mathbb C}}

\newcommand{\OO}{\mathcal{O}}
\newcommand{\oo}{o}

\newcommand{\be}{\begin{equation}}
\newcommand{\ee}{\end{equation}}

\newcommand{\ti}{\tilde}

\newcommand{\E}{\mathrm{e}}
\newcommand{\I}{\mathrm{i}}

\newcommand{\im}{\mathrm{Im}}
\newcommand{\re}{\mathrm{Re}}

\newcommand{\ceil}[1]{\lceil#1 \rceil}

\newcommand{\eps}{\varepsilon}

\newcommand{\sig}{\sigma}
\newcommand{\lam}{\lambda}

\numberwithin{equation}{section}

\begin{document}

\title[Uniqueness Results for Schr\"odinger Operators]{Uniqueness Results for One-Dimensional Schr\"odinger Operators with Purely Discrete Spectra}

\author[J.\ Eckhardt]{Jonathan Eckhardt}
\address{Faculty of Mathematics\\ University of Vienna\\
Nordbergstrasse 15\\ 1090 Wien\\ Austria}
\email{\mailto{jonathan.eckhardt@univie.ac.at}}
\urladdr{\url{http://homepage.univie.ac.at/jonathan.eckhardt/}}

\author[G.\ Teschl]{Gerald Teschl}
\address{Faculty of Mathematics\\ University of Vienna\\
Nordbergstrasse 15\\ 1090 Wien\\ Austria\\ and International
Erwin Schr\"odinger
Institute for Mathematical Physics\\ Boltzmanngasse 9\\ 1090 Wien\\ Austria}
\email{\mailto{Gerald.Teschl@univie.ac.at}}
\urladdr{\url{http://www.mat.univie.ac.at/~gerald/}}

\thanks{Trans.\ Amer.\ Math.\ Soc.\ {\bf 365}, 3923--3942 (2013)}
\thanks{{\it Research supported by the Austrian Science Fund (FWF) under Grant No.\ Y330}}


\begin{abstract}
We provide an abstract framework for singular one-dimensional Schr\"odinger operators  with purely discrete spectra
to show when the spectrum plus norming constants determine such an operator completely. As an example we apply
our findings to prove new uniqueness results for perturbed quantum mechanical harmonic oscillators. In addition, we
also show how to establish a Hochstadt--Lieberman type result for these operators. Our approach is based on
the singular Weyl--Titchmarsh--Kodaira theory which is extended to cover the present situation.
\end{abstract}

\maketitle

\section{Introduction}
\label{sec:int}

The present paper is concerned with uniqueness results for one-dimensional Schr\"o\-dinger operators
\be
H = -\frac{d^2}{dx^2} + q(x), \qquad x\in (a,b),
\ee
on the Hilbert space $L^2(a,b)$ with a real-valued potential $q\in L^1_{\mathrm{loc}}(a,b)$. We are particularly interested
in the case where $H$ has purely discrete spectrum. Of course this problem is well understood in the case where the
operator is regular, that is, $(a,b)$ is compact and $q\in L^1(a,b)$, but for singular operators there are still many
open questions. One of the prime examples in this respect are perturbations of the quantum mechanical oscillator and in
particular its isospectral class \cite{le}, \cite{mktr}. In particular, perturbations 
\be\label{eqopperho}
H = -\frac{d^2}{dx^2} + x^2 + q(x), \qquad x\in (-\infty,\infty),
\ee
of the harmonic oscillator have attracted much interest recently; see Chelkak, Kargaev and Korotyaev \cite{chelkak}, \cite{ckk}, \cite{ckk2}
and the references therein.

Moreover, it has been shown by Kodaira \cite{ko}, Kac \cite{ka} and more recently by Fulton \cite{ful08}, Gesztesy and Zinchenko \cite{gz},
Fulton and Langer \cite{fl}, Kurasov and Luger \cite{kl}, and Kostenko, Sakhnovich, and Teschl \cite{kt}, \cite{kst}, \cite{kst2}, \cite{kst3}
that, for a large class of singularities at $a$, it is still possible to define a singular Weyl function at the base point $a$.
While in these previous works the main focus was on applications to spherical Schr\"odinger operators
\be\label{eqopperbe}
H = -\frac{d^2}{dx^2} + \frac{l(l+1)}{x^2} + q(x), \qquad x\in (0,\infty),
\ee
(also known as Bessel operators), our interest here will be to apply these techniques to operators of the form \eqref{eqopperho}.

For a different approach to uniqueness results for one-dimensional Schr\"odinger operators based on Krein's spectral shift
function we refer to \cite{gs1}.

The outline of our paper is as follows: In Section~\ref{sec:swm} we will fix our notation and recall some basic facts from
singular Weyl--Titchmarsh theory. In \cite{kst2} the authors have proven a local Borg--Marchenko theorem in the case
where the spectrum created by the singular endpoint has convergence exponent less than one.  While this was
sufficient to cover Bessel-type operators \eqref{eqopperbe}, it is not good enough for \eqref{eqopperho} where the
convergence exponent will be one. Hence our first aim will be to extend the results from \cite{kst2} to arbitrary
(finite) growth orders in Section~\ref{sec:egr} such that we can provide an associated local Borg--Marchenko theorem
in Section~\ref{sec:lbmt}. In Section~\ref{sec:urds} we will then use this to prove uniqueness results for operators
with purely discrete spectrum. We will provide a general result which shows that the spectrum together with the
norming constants uniquely determines the operator. As a special case we will obtain a (slight) generalization of the
main result from \cite{ckk}.

The Borg--Marchenko uniqueness theorem is of course also the main ingredient in a vast number of other
uniqueness results in inverse spectral theory. One of these results which had particular impact is the celebrated
Hochstadt--Lieberman theorem \cite{hl}. Hence we will try to use this direction as a test case for our results and prove a
powerful generalization of this famous theorem to singular operators with discrete spectra. In fact, while
many extensions are known to date, we refer to \cite{sa}, \cite{ho}, \cite{hm} for recent accounts, most of them
concern regular operators (including the case where the potential is a distribution) and we are only aware of two references
dealing with singular operators. First of all the work by Gesztesy and Simon \cite{gs2}, who considered the case of operators which
grow faster than the harmonic operator and satisfy $q(-x)\ge q(x)$. Secondly Khodakovsky \cite{kh}, who improved their
result and removed the growth restriction. However, there are many interesting physical examples which are not
covered by this result. For example, P\"oschl--Teller type potentials, which have non-integrable singularities near the
endpoints, or perturbations of the harmonic oscillator \eqref{eqopperho}. We will show in Sections~\ref{sec:urds}
and \ref{sec:pho} that our result is able to cover these examples.

\section{Singular Weyl--Titchmarsh theory}
\label{sec:swm}

Our fundamental ingredient will be singular Weyl--Titchmarsh theory and hence we begin by recalling
the necessary facts from \cite{kst2}.
To set the stage, we will consider one-dimensional Schr\"odinger operators on $L^2(a,b)$
with $-\infty \le a<b \le \infty$ of the form
\begin{equation} \label{stli}
\tau = - \frac{d^2}{dx^2} + q(x),
\end{equation}
where the potential $q$ is real-valued satisfying
\begin{equation}
q \in L^1_{loc}(a,b).
\end{equation}
We will use $\tau$ to denote the formal differential expression and $H$ to denote a corresponding
self-adjoint operator given by $\tau$ with separated boundary conditions at $a$ and/or $b$.

We will choose a point $c\in(a,b)$ and also consider the operators $H^D_{(a,c)}$, $H^D_{(c,b)}$
which are obtained by restricting $H$ to $(a,c)$, $(c,b)$ with a Dirichlet boundary condition at
$c$, respectively. The corresponding operators with a Neumann boundary condition will be
denoted by $H^N_{(a,c)}$ and $H^N_{(c,b)}$. 
Moreover, let $c(z,x)$, $s(z,x)$ be the solutions of $\tau u = z\, u$ corresponding
to the initial conditions $c(z,c)=1$, $c'(z,c)=0$ and $s(z,c)=0$, $s'(z,c)=1$. 
Define the Weyl functions (corresponding to the base point $c$) such that
\begin{subequations}
\begin{align}
u_-(z,x) &= c(z,x) - m_-(z) s(z,x), \qquad z\in\C\setminus\sig(H^D_{(a,c)}),\\\label{defupm}
u_+(z,x) &= c(z,x) + m_+(z) s(z,x), \qquad z\in\C\setminus\sig(H^D_{(c,b)}),
\end{align}
\end{subequations}
are square integrable near $a$, $b$ and satisfy the boundary condition of $H$
at $a$, $b$ (if any), respectively. The solutions $u_\pm(z,x)$ (as well as their multiples)
are called Weyl solutions at $a$, $b$.
For further background we refer to \cite[Chap.~9]{tschroe} or \cite{wdln}.

To define an analogous singular Weyl function at the, in general singular, endpoint $a$ we will
first need the analog of the system of solutions $c(z,x)$ and $s(z,x)$. Hence our first goal is
to find a system of entire solutions $\theta(z,x)$ and $\phi(z,x)$ such that $\phi(z,x)$ lies in the domain of
$H$ near $a$ and such that the Wronskian
\be
W(\theta,\phi)= \theta(z,x) \phi'(z,x) - \theta'(z,x)\phi(z,x)=1.
\ee
To this end we start with a
hypothesis which will turn out necessary and sufficient for such a system of solutions to exist.

\begin{hypothesis}\label{hyp:gen}
Suppose that the spectrum of $H^D_{(a,c)}$ is purely discrete for one (and hence for all)
$c\in (a,b)$.
\end{hypothesis}

Note that this hypothesis is for example satisfied if $q(x) \to +\infty$ as $x\to a$ (cf.\ Problem~9.7 in \cite{tschroe}).

\begin{lemma}[\cite{kst2}]\label{lem:pt}
The following properties are equivalent:
\begin{enumerate}
\item The spectrum of $H_{(a,c)}^D$ is purely discrete for some $c\in(a,b)$.
\item There is a real entire solution $\phi(z,x)$, which is non-trivial and lies in the domain of $H$ near $a$ for each $z\in\C$.
\item There are real entire solutions $\phi(z,x)$, $\theta(z,x)$ with $W(\theta,\phi)=1$, such that
 $\phi(z,x)$ is non-trivial and lies in the domain of $H$ near $a$ for each $z\in\C$.
\end{enumerate}
\end{lemma}

\begin{remark}\label{rem:uniq}
It is important to point out that such a fundamental system is not unique and any other such system is given by
\[
\ti{\theta}(z,x) = \E^{-g(z)} \theta(z,x) - f(z) \phi(z,x), \qquad
\ti{\phi}(z,x) = \E^{g(z)} \phi(z,x),
\]
where $f(z)$, $g(z)$ are entire functions with $f(z)$ real and $g(z)$ real modulo $\I\pi$.
The singular Weyl functions are related via
\[
\ti{M}(z) = \E^{-2g(z)} M(z) + \E^{-g(z)}f(z).
\]
\end{remark}

We will need the following simple lemma on the high energy asymptotics of the solution $\phi(z,x)$.
 Note that we always use the principal square root with branch cut along the negative real axis.

\begin{lemma}\label{lemPhiAs}
If $\phi(z,x)$ is a real entire solution which lies in the domain of $H$ near $a$, then for every $x_0$, $x\in(a,b)$
\be\label{IOphias}
\phi(z,x) = \phi(z,x_0) \E^{(x-x_0) \sqrt{-z}} \left(1 + \OO\left(1/\sqrt{-z}\right)\right),
\ee
as $|z|\to\infty$ along any nonreal ray.
\end{lemma}

\begin{proof}
Using
\[
\phi(z,x) = \phi(z,c) ( c(z,x) - m_-(z) s(z,x) )
\]
and the well-known asymptotics of $c(z,x)$, $s(z,x)$, and $m_-(z)$ (cf.\ \cite[Lemma~9.18 and Lemma~9.19]{tschroe})
we see~\eqref{IOphias} for $x_0=c$ and $x>x_0$. The case $x<x_0$ follows after reversing the roles of $x_0$ and $x$.
Since $c$ is arbitrary, the proof is complete.
\end{proof}

Given a system of real entire solutions $\phi(z,x)$ and $\theta(z,x)$ as in Lemma~\ref{lem:pt} we can define the
singular Weyl function
\be\label{defM}
M(z) = -\frac{W(\theta(z),u_+(z))}{W(\phi(z),u_+(z))}
\ee
such that the solution which is in the domain of $H$ near $b$ (cf.\ \eqref{defupm}) is given by
\be
u_+(z,x)= a(z) \big(\theta(z,x) + M(z) \phi(z,x)\big),
\ee
where $a(z)= - W(\phi(z),u_+(z))$.
By construction we obtain that the singular Weyl function $M(z)$ is analytic in $\C\backslash\R$ and satisfies $M(z)=M(z^*)^*$.
Rather than $u_+(z,x)$ we will use
\be\label{defpsi}
\psi(z,x)= \theta(z,x) + M(z) \phi(z,x).
\ee 
Recall also from \cite[Lemma~3.2]{kst2} that associated with $M(z)$ is a corresponding spectral measure $\rho$ by virtue of
the Stieltjes--Liv\v{s}i\'{c} inversion formula
\be\label{defrho}
\frac{1}{2} \left( \rho\big((x_0,x_1)\big) + \rho\big([x_0,x_1]\big) \right)=
\lim_{\eps\downarrow 0} \frac{1}{\pi} \int_{x_0}^{x_1} \im\big(M(x+\I\eps)\big) dx.
\ee

\begin{theorem}[\cite{gz}]
Define
\be
\hat{f}(\lam) = \lim_{c\uparrow b} \int_a^c \phi(\lam,x) f(x) dx,
\ee
where the right-hand side is to be understood as a limit in $L^2(\R,d\rho)$. Then the map
\be
U: L^2(a,b) \to L^2(\R,d\rho), \qquad f \mapsto \hat{f},
\ee
is unitary and its inverse is given by
\be\label{Uinv}
f(x) = \lim_{r\to\infty} \int_{-r}^r \phi(\lam,x) \hat{f}(\lam) d\rho(\lam),
\ee
where again the right-hand side is to be understood as a limit in $L^2(a,b)$.
Moreover, $U$ maps $H$ to multiplication with $\lam$.
\end{theorem}

\begin{remark}
We have seen in Remark~\ref{rem:uniq} that $M(z)$ is not unique. However, given
$\ti{M}(z)$ as in Remark~\ref{rem:uniq}, the spectral measures are related by
\[
d\ti{\rho}(\lam) = \E^{-2g(\lam)} d\rho(\lam).
\]
Hence the measures are mutually absolutely continuous and the associated spectral
transformation just differ by a simple rescaling with the positive function $\E^{-2g(\lam)}$.
\end{remark}

Finally, $M(z)$ can be reconstructed from $\rho$ up to an entire function via the following integral representation.

\begin{theorem}[\cite{kst2}]\label{IntR}
Let $M(z)$ be a singular Weyl function and $\rho$ its associated spectral measure. Then there exists
an entire function $g(z)$ such that $g(\lam)\ge 0$ for $\lam\in\R$ and $\E^{-g(\lam)}\in L^2(\R, d\rho)$.

Moreover, for any entire function $\hat{g}(z)$ such that $\hat{g}(\lam)>0$ for $\lam\in\R$ and $(1+\lam^2)^{-1} \hat{g}(\lam)^{-1}\in L^1(\R, d\rho)$
(e.g.\ $\hat{g}(z)=\E^{2g(z)}$) we have the integral representation
\be\label{Mir}
M(z) = E(z) + \hat{g}(z) \int_\R \left(\frac{1}{\lam-z} - \frac{\lam}{1+\lam^2}\right) \frac{d\rho(\lam)}{\hat{g}(\lam)},
\qquad z\in\C\backslash\sig(H),
\ee
where $E(z)$ is a real entire function.
\end{theorem}

\begin{remark}\label{rem:herg}
Choosing a real entire function $g(z)$ such that $\exp(-2g(\lam))$ is integrable with respect to $d\rho$, we see that
\be
M(z) = \E^{2g(z)} \int_\R \frac{1}{\lam-z} \E^{-2g(\lam)}d\rho(\lam) - E(z).
\ee
Hence if we choose $f(z) = \exp(-g(z)) E(z)$ and switch to a new system of solutions as in Remark~\ref{rem:uniq},
we see that the new singular Weyl function is a Herglotz--Nevanlinna function
\be
\ti{M}(z) = \int_\R \frac{1}{\lam-z} \E^{-2g(\lam)}d\rho(\lam).
\ee
\end{remark}

\section{Exponential growth rates}
\label{sec:egr}

While a real entire fundamental system $\theta(z,x)$, $\phi(z,x)$ as in Section~\ref{sec:swm} is sufficient to define a singular
Weyl function and an associated spectral measure, it does not suffice for the proof of our uniqueness results. 
For them we will need information on the growth order of the solutions $\theta(\,\cdot\,,x)$ and $\phi(\,\cdot\,,x)$. In the case
where $a$ is finite with a repelling potential, the growth rate will be $\nicefrac{1}{2}$ and this case was dealt
with in \cite[Section~6]{kst2}. The aim of the present section is to extend these results to cover arbitrary (finite) growth rates.
 Therefore we will say a real entire solution $\phi(z,x)$ is of growth order at most $s>0$ if both $\phi(\,\cdot\,,x)$ and $\phi'(\,\cdot\,,x)$
are of growth order at most $s$ for all $x\in(a,b)$.

Our first aim is to extend Lemma~\ref{lem:pt} and to show how the growth order of $\phi(\,\cdot\,,x)$ is connected with the convergence exponent
of the spectrum. To this end we begin by recalling some basic notation. We refer to the classical book by Levin \cite{lev} for proofs and further
background.

Given some discrete set $S\subseteq\C$, the number
\be
 \inf\biggr\lbrace s\geq0 \,\biggr|\, \sum_{\mu\in S} \frac{1}{1+|\mu|^s}<\infty \biggr\rbrace \in [0,\infty],
\ee
is called the convergence exponent of $S$. Moreover, the smallest integer $p\in\N$ for which
\be
 \sum_{\mu\in S} \frac{1}{1+|\mu|^{p+1}}<\infty
\ee
will be called the genus of $S$. Introducing the elementary factors
\be
 E_p(\zeta,z) = \left(1-\frac{z}{\zeta}\right) \exp\left(\sum_{k=1}^p\frac{1}{k} \frac{z^k}{\zeta^k}\right), \quad z\in\C,
\ee
if $\zeta\not=0$ and $E_p(0,z)=z$ we recall that the product $\prod_{\mu\in S} E_p(\mu,z)$
converges uniformly on compact sets to an entire function of growth order $s$, where $s$ and $p$ are the
convergence exponent and genus of $S$, respectively.

Furthermore, we will denote the spectrum of $H_{(a,c)}^D$ and $H_{(a,c)}^N$ (provided they are discrete) by
\be
\sig(H_{(a,c)}^D) = \{ \mu_n(c) \}_{n\in N}, \qquad \sig(H_{(a,c)}^N) = \{ \nu_{n-1}(c) \}_{n\in N},
\ee
where the index set $N$ is either $\N$ or $\Z$. The eigenvalues $\mu_n(c)$, $\nu_n(c)$ are precisely the zeros
of $\phi(\,\cdot\,,c)$ and $\phi'(\,\cdot\,,c)$, respectively. Recall that both spectra are interlacing
\be\label{IOinterlacing}
\nu_{n-1}(c) < \mu_n(c) < \nu_n(c), \qquad n\in N,
\ee
and that Krein's theorem \cite[Theorem~27.2.1]{lev} states
\be\label{eqKrein}
m_-(z) = C \prod_{j\in N} \frac{E_0(\nu_{n-1}(c),z)}{E_0(\mu_n(c),z)}, \quad C\ne 0.
\ee
Note that in general the products in the numerator and denominator will not converge independently, but
only jointly, since due to the interlacing properties of the eigenvalues, the sum
\[
\sum_{n\in N} \left( \frac{1}{\nu_{n-1}(c)} -  \frac{1}{\mu_n(c)}\right)
\]
will converge.

\begin{theorem}\label{thm:IOphiev}
For each $s>0$ the following properties are equivalent:
\begin{enumerate}
\item The spectrum of $H_{(a,c)}^D$ is purely discrete and has convergence exponent at most $s$.
\item There  is a real entire solution $\phi(z,x)$ of growth order at most $s$ which is non-trivial and lies in the domain of $H$ near $a$ for each $z\in\C$.
\end{enumerate} 
In this case $s\geq\frac{1}{2}$.
\end{theorem}

\begin{proof}
Suppose the spectrum of $H_{(a,c)}^D$ is purely discrete and has convergence exponent at most $s$. The same then
holds true for the spectrum of $H_{(a,c)}^N$ and according to~\cite[Lemma~6.3]{kst2}, $s$ is at least $\nicefrac{1}{2}$.
Denote by $p\in\N_0$ the genus of these sequences and consider the real entire functions
\[
\alpha(z) = \prod_{n\in N} E_p\left(\mu_n(c),z\right) \quad\text{and}\quad \ti{\beta}(z) = \prod_{n\in N} E_p\left(\nu_{n-1}(c),z\right).
\]
Then $\alpha(z)$ and $\ti{\beta}(z)$ are of growth order at most $s$ by Borel's theorem (see \cite[Theorem~4.3.3]{lev}). Next note that
\[
m_-(z) = \E^{h(z)} \frac{\ti{\beta}(z)}{\alpha(z)}
\]
for some entire function $h(z)$ since the right-hand side has the same poles and zeros as $m_-(z)$. Comparing this with Krein's formula
\eqref{eqKrein} we obtain that $h(z)$ is in fact a polynomial of degree at most $p$:
\[
h(z) = \sum_{k=1}^p \frac{z^k}{k} \sum_{n\in N} \left( \frac{1}{\mu_n(c)^k} - \frac{1}{\nu_{n-1}(c)^k}\right) +\log(C).
\]
Observe that the sums converge absolutely by our interlacing assumption. In particular,
\be
\beta(z) = -m_-(z)\alpha(z)= - \E^{h(z)} \ti{\beta}(z)
\ee
is of growth order at most $s$ as well. Hence the solutions
\begin{align*}
 \phi(z,x) = \alpha(z)s(z,x) + \beta(z)c(z,x), \quad x\in(a,b),~z\in\C,
\end{align*}
lie in the domain of $H$ near $a$ and are of growth order at most $s$ by~\cite[Lemma~9.18]{tschroe}. 

Conversely let $\phi(z,x)$ be a real entire solution of growth order at most $s$ which lies in the domain of $H$ near $a$.
Then since $m_-(z) = -\phi'(z,c)/\phi(z,c)$, the spectrum of $H_{(a,c)}^D$ is purely discrete and coincides with the zeros of $\phi(\,\cdot\,,c)$.
Now since $\phi(\,\cdot\,,c)$ is of growth order at most $s$, its zeros are of convergence exponent at most $s$.
\end{proof}

Note that because of the interlacing property of eigenvalues, it is irrelevant what boundary condition we choose at the point $c$.
Moreover, the preceding theorem also shows that the convergence exponent of $\sigma(H_{(a,c)}^D)$ is independent of $c\in(a,b)$.
Finally, note that for (i) to hold, it suffices that there is some real entire solution $\phi(z,x)$ such that $\phi(\,\cdot\,,x)$ is of growth order at most $s$ for some $x\in(a,b)$.

Unfortunately, given a real entire solution $\phi(z,x)$ of growth order $s>0$ we are not able to prove the existence of a second solution of the same growth order.
However, at least under some additional assumptions we get a second solution $\theta(z,x)$ of growth order arbitrarily close to $s$.
For the proof we will need the following version of the corona theorem for entire functions.

\begin{theorem}[\cite{hor}]\label{thm:cor}
Let $R_s(\C)$, $s> 0$, be the ring of all entire functions $f(z)$ for which there are constants $A$, $B >0$ such that
\be
|f(z)| \le B \E^{A |z|^s}, \quad z\in\C.
\ee
Then $f_j\in R_s(\C)$, $j=1,\dots,n$ generate $R_s(\C)$ if and only if
\be
|f_1(z)| + \cdots + |f_n(z)| \ge b \E^{-a |z|^s}, \quad z\in\C,
\ee
for some constants $a$, $b>0$.
\end{theorem}

As an immediate consequence of this result, we obtain the following necessary and sufficient criterion
for the existence of a second solution of the required type.

\begin{lemma}
Suppose $\phi(z,x)$ is a real entire solution of growth order $s>0$ and let $\eps>0$. 
Then there is a real entire second solution $\theta(\,\cdot\,,x)\in R_{s+\eps}(\C)$ with $W(\theta,\phi)=1$ if and only if 
\begin{align}\label{eqnSSphiboundlow}
   |\phi(z,c)| + |\phi'(z,c)| \geq b \E^{-a|z|^{s+\eps}}, \quad z\in\C,
\end{align}
for some constants $a$, $b>0$.
\end{lemma}

\begin{proof}
If $\theta(z,x)$ is a second solution which lies in $R_{s+\eps}(\C)$, then
\begin{align*}
   \theta(z,c) \phi'(z,c) - \theta'(z,c) \phi(z,c) = 1, \quad z\in\C,
\end{align*}
implies that the functions $\phi(\,\cdot\,,c)$ and $\phi'(\,\cdot\,,c)$ generate $R_{s+\eps}(\C)$ and~\eqref{eqnSSphiboundlow} follows from Theorem~\ref{thm:cor}.
Conversely if~\eqref{eqnSSphiboundlow} holds, then $\phi(\,\cdot\,,c)$ and $\phi'(\,\cdot\,,c)$ generate $R_{s+\eps}(\C)$. 
Thus there are real entire functions $\gamma$, $\delta\in R_{s+\eps}(\C)$ with
\begin{align*}
   \gamma(z) \phi'(z,c) - \delta(z) \phi(z,c) = 1, \quad z\in\C.
\end{align*}
Now take $\theta(z,x)$  to be the solutions with initial conditions $\theta(z,c)=\gamma(z)$ and $\theta'(z,c)=\delta(z)$. 
\end{proof}

We are also able to provide a sufficient condition for a second solution of order to exist, in terms of the zeros of $\phi(\,\cdot\,,c)$ and $\phi'(\,\cdot\,,c)$.
For the proof we need the following lemma on the minimal modulus of an entire function of finite growth order.

\begin{lemma}\label{lemesthp}
Suppose $F(z)$ is an entire function of growth order $s$ with zeros $\zeta_j$, $j\in\N$. Then for each $\delta$, $\eps>0$ there are constants $A$, $B>0$ such that
\begin{align}\label{estcp}
\left|F(z) \right| \ge B \E^{-A |z|^{s+\eps}},
\end{align}
except possibly when $z$ belongs to one of the disks $|z-\zeta_j|<|\zeta_j|^{-\delta}$.
\end{lemma}

\begin{proof}
This follows from Hadamard's factorization theorem and \cite[Lemma~2.6.18]{boas}.
\end{proof}

\begin{lemma}\label{lem:cor}
Suppose $\phi(z,x)$ is a real entire solution of growth order $s>0$ and that for some $r>0$ all but finitely many of the disks given by
\be\label{IOestmunu}
|z-\mu_n(c)|<|\mu_n(c)|^{-r} \quad\text{and}\quad |z-\nu_{n-1}(c)|<|\nu_{n-1}(c)|^{-r}, \quad n\in N,
\ee
are disjoint. Then for every $\eps>0$ there is a real entire second solution $\theta(z,x)$ of growth order at most $s+\eps$ and $W(\theta,\phi)=1$.
\end{lemma}

\begin{proof}
Lemma~\ref{lemesthp} implies that we have~\eqref{estcp} for either $\phi(\,\cdot\,,c)$ or $\phi'(\,\cdot\,,c)$ and hence in particular for the sum of both.
\end{proof}

\begin{remark}\label{rem:uniqExp}
By the Hadamard product theorem~\cite[Theorem~4.2.1]{lev}, a solution $\phi(z,x)$ of growth order $s>0$ is unique up to a factor $\E^{g(z)}$,
for some polynomial $g(z)$ real modulo $\I\pi$ and of degree at most $p$, where $p\in\N_0$ is the genus of the eigenvalues of $H_{(a,c)}^D$.
A solution $\theta(z,x)$ of growth order at most $s$ is unique only up to $f(z) \phi(z,x)$, where $f(z)$ is an entire function of growth order at most $s$.
\end{remark}

Finally, note that under the assumptions in this section one can use $\hat{g}(z)=\exp(z^{2\ceil{(p+1)/2}})$ in Theorem~\ref{IntR}.
If in addition $H$ is bounded from below, then one can also use $\hat{g}(z)=\exp(z^{p+1})$.

\section{A local Borg--Marchenko uniqueness result}
\label{sec:lbmt}

The purpose of the present section is again to extend the corresponding results from \cite[Section~7]{kst2} to the case of arbitrary
growth orders.

\begin{lemma}[\cite{kst2}]\label{lemAsymM}
For each $x\in(a,b)$, the singular Weyl function $M(z)$ and the Weyl solution $\psi(z,x)$ defined in~\eqref{defpsi}
have the following asymptotics:
\begin{align}\label{asymM}
M(z) &= -\frac{\theta(z,x)}{\phi(z,x)} + \OO\left(\frac{1}{\sqrt{-z}\phi(z,x)^2}\right),\\ \label{asympsi}
\psi(z,x) &= \frac{1}{2\sqrt{-z} \phi(z,x)} \left( 1 + \OO\left(\frac{1}{\sqrt{-z}}\right) \right),
\end{align}
as $|z|\to\infty$ in any sector $|\im(z)| \geq \delta\, |\re(z)|$.
\end{lemma}

In particular, \eqref{asymM} shows that asymptotics of $M(z)$ immediately follow once one
has corresponding asymptotics for the solutions $\theta(z,x)$ and $\phi(z,x)$. Moreover, the leading
asymptotics depend only on the values of $q$ near the endpoint $a$ (and on the choice of $\theta(z,x)$ and $\phi(z,x)$). 
The following Borg--Marchenko type uniqueness result shows that the converse is also true.

In order to state this theorem, let $q_0$ and $q_1$ be two potentials on intervals $(a,b_0)$ and $(a,b_1)$, respectively.
By $H_0$ and $H_1$ we denote some corresponding self-adjoint operators with separated boundary conditions.
Furthermore, for $j=0,1$, let $\theta_j(z,x)$, $\phi_j(z,x)$ be some real entire fundamental system of solutions with
$W(\theta_j,\phi_j)=1$ such that $\phi_j(z,x)$ lies in the domain of $H_j$ near $a$. 
The associated singular Weyl functions are denoted by $M_0(z)$ and $M_1(z)$. We will also use
the common short-hand notation $\phi_1(z,x) \sim \phi_2(z,x)$ to abbreviate the asymptotic relation
$\phi_1(z,x) = \phi_2(z,x) (1 + \oo(1))$ (or equivalently $\phi_2(z,x) = \phi_1(z,x) (1 + \oo(1))$) as $|z|\to\infty$
in some specified manner.

\begin{theorem}\label{thmbm}
Suppose $\theta_0(z,x)$, $\theta_1(z,x)$, $\phi_0(z,x)$, $\phi_1(z,x)$ are of growth order at most $s$ for some $s>0$ and $\phi_1(z,x) \sim \phi_0(z,x)$ for one (and hence by~\eqref{IOphias} for all) $x\in(a,b_0)\cap(a,b_1)$ as $|z|\to\infty$ along some nonreal rays dissecting the complex plane into sectors of opening angles less than $\nicefrac{\pi}{s}$.
Then for each $c\in(a,b_0)\cap(a,b_1)$, the following properties are equivalent:
\begin{enumerate}
 \item We have $q_0(x) = q_1(x)$ for almost all $x\in(a,c)$ and $W(\phi_0,\phi_1)(a)=0$.
 \item For each $\delta>0$ there is an entire function $f(z)$ of growth order at most $s$ such that
  \begin{align*}M_1(z)-M_0(z) = f(z) + \OO\left(\frac{1}{\sqrt{-z} \phi_0(z,c)^{2}}\right),\end{align*}
   as $|z|\rightarrow\infty$ in the sector $|\im(z)|\geq \delta\,|\re(z)|$.
 \item For each $d\in(a,c)$ there is an entire function $f(z)$ of growth order at most $s$ such that
  \begin{align*}M_1(z)-M_0(z) = f(z) + \OO\left(\frac{1}{\phi_0(z,d)^{2}}\right),\end{align*}
   as $|z|\rightarrow\infty$ along our nonreal rays.
\end{enumerate}
\end{theorem}

\begin{proof}
If (i) holds, then by Remark~\ref{rem:uniqExp} the solutions are related by
\begin{align}\label{eqnLBMsolPHI}
 \phi_1(z,x) = \phi_0(z,x)\E^{g(z)}, \quad x\in(a,c],~ z\in\C,
\end{align}
and
\begin{align}\label{eqnLBMsolTHETA}
 \theta_1(z,x) = \theta_0(z,x) \E^{-g(z)} -f(z)\phi_1(z,x), \quad x\in(a,c],~z\in\C,
\end{align}
 for some polynomial $g(z)$ of degree at most $s$ and some real entire function $f(z)$ of growth order at most $s$. 
 From the asymptotic behavior of the solutions $\phi_0(z,x)$, $\phi_1(z,x)$ we infer that $g=0$.
 Now the asymptotics in Lemma~\ref{lemAsymM} show that
 \begin{align*}
  M_1(z) - M_0(z) & = \frac{\theta_0(z,c)}{\phi_0(z,c)} - \frac{\theta_1(z,c)}{\phi_1(z,c)} + \OO\left(\frac{1}{\sqrt{-z}\phi_0(z,c)^{2}}\right) \\
                  & = f(z) + \OO\left(\frac{1}{\sqrt{-z}\phi_0(z,c)^{2}}\right),
 \end{align*}
 as $|z|\to\infty$ in any sector $|\im(z)| \ge \delta\, |\re(z)|$.
The implication (ii) $\Rightarrow$ (iii) is obvious.
Now suppose property (iii) holds and for each fixed $x\in(a,c)$ consider the entire function
\begin{align}\label{eqnthmbmentfun}
 G_x(z) = \phi_1&(z,x)  \theta_0(z,x) - \phi_0(z,x) \theta_1(z,x) - f(z)\phi_0(z,x) \phi_1(z,x), \quad z\in\C.
\end{align}
Since away from the real axis this function may be written as
\begin{align*}
 G_x(z) & = \phi_1(z,x) \psi_0(z,x) - \phi_0(z,x) \psi_1(z,x) \\
      & \qquad\qquad + (M_1(z)-M_0(z) - f(z)) \phi_0(z,x) \phi_1(z,x), \quad z\in\C\backslash\R,
\end{align*}
it vanishes as $|z|\to\infty$ along our nonreal rays. For the first two terms this
follows from \eqref{asympsi} together with our hypothesis that $\phi_0(\,\cdot\,,x)$ and $\phi_1(\,\cdot\,,x)$
have the same asymptotics. The last term tends to zero because of our assumption on the difference of the Weyl functions.
Moreover, by our hypothesis $G_x$ is of growth order at most $s$ and thus we
can apply the Phragm\'en--Lindel\"of theorem (e.g., \cite[Section~6.1]{lev}) in the sectors bounded by our rays.
This shows that $G_x$ is bounded on all of $\C$.
By Liouville's theorem it must be constant and since it vanishes along a ray, it must be zero; that is,
\begin{align*}
\phi_1(z,x) \theta_0(z,x) - \phi_0(z,x) \theta_1(z,x) = f(z)\phi_0(z,x)\phi_1(z,x), \quad x\in(a,c),~z\in\C.
\end{align*}
Dividing both sides of this identity by $\phi_0(z,x)\phi_1(z,x)$, differentiating with respect to $x$, and using $W(\theta_j,\phi_j)=1$ shows
$\phi_1(z,x)^2 = \phi_0(z,x)^2$. Taking the logarithmic derivative further gives
$\phi_1'(z,x)/\phi_1(z,x)=\phi_0'(z,x)/\phi_0(z,x)$, which shows $W(\phi_0,\phi_1)(a)=0$. Differentiating once more shows
$\phi_1''(z,x)/\phi_1(z,x)=\phi_0''(z,x)/\phi_0(z,x)$. This finishes the proof since $q_j(x)=z + \phi_j''(z,x)/\phi_j(z,x)$.
\end{proof}

Note that the implication (iii) $\Rightarrow$ (i) could also be proved under somewhat weaker conditions.
First of all the assumption on the growth of the entire functions $f(z)$ is only due to the use of the Phragm\'{e}n--Lindel\"{o}f principle.
Hence it would also suffice that for each $\eps>0$ we have
\begin{align}\label{eqnbmaltgrowth}
\sup_{|z|=r_n} |f(z)| \leq B \E^{A r_n^{s+\eps}},
\end{align}
for some increasing sequence of positive numbers $r_n\uparrow\infty$ and constants $A$, $B\in\R$.
Furthermore, for this implication to hold it would also suffice that the solutions have the same order of magnitude as $|z|\rightarrow\infty$ along our nonreal rays instead of the same asymptotics.

While at first sight it might look like the condition on the asymptotics of the solutions $\phi_j(z,x)$ requires knowledge
about them, this is not the case, since the high energy asymptotics will only involve some qualitative information
on the kind of the singularity at $a$ as we will show in Section~\ref{secPQHO}.
Next, the appearance of the additional freedom of the function $f(z)$ just reflects the fact that we only ensure the same normalization
for the solutions $\phi_0(z,x)$ and $\phi_1(z,x)$ but not for $\theta_0(z,x)$ and $\theta_1(z,x)$ (cf.\ Remark~\ref{rem:uniqExp}).

\begin{corollary}\label{corbm}
Suppose $\theta_0(z,x)$, $\theta_1(z,x)$, $\phi_0(z,x)$, $\phi_1(z,x)$ are of growth order at most $s$ for some $s>0$ and  $\phi_1(z,x) \sim \phi_0(z,x)$ for some $x\in(a,b_0)\cap(a,b_1)$ as $|z|\to\infty$ along some nonreal rays dissecting the complex plane into sectors of opening angles less than $\nicefrac{\pi}{s}$.
If 
\begin{align}\label{eqncorbm}
 M_1(z) - M_0(z) = f(z), \quad z\in\C\backslash\R,
\end{align}
for some entire function $f(z)$ of growth order at most $s$, then $H_0=H_1$.
\end{corollary}

\begin{proof}
Without loss of generality, suppose that $b_0\leq b_1$.
Theorem~\ref{thmbm} shows that $q_0(x)=q_1(x)$ for almost all  $x\in(a,b_0)$ and that the boundary condition at $a$ (if any) is the same.
As in the proof of Theorem~\ref{thmbm} one sees that~\eqref{eqnLBMsolPHI} and~\eqref{eqnLBMsolTHETA} hold, hence
\begin{align*}
  \psi_1(z,x) & = \theta_1(z,x) + M_1(z)\phi_1(z,x) \\
              & = \theta_0(z,x) - f(z)\phi_1(z,x) + (M_0(z)+f(z))\phi_0(z,x) \\
              & = \theta_0(z,x) + M_0(z)\phi_0(z,x) = \psi_0(z,x), 
\end{align*}
for each $x\in(a,b_0)$, $z\in\C\backslash\R$.
If $b_0<b_1$, then the right endpoint $b_0$ of $H_0$ were regular as $q_1$ is integrable over $[c,b_0]$.
Thus $\psi_0(z,x)$ and hence also $\psi_1(z,x)$ would satisfy some boundary condition at $b_0$.
Since this is not possible, we necessarily have $b_0=b_1$. 
Finally since $\psi_0(z,x)=\psi_1(z,x)$, $H_0$ and $H_1$ have the same boundary condition at $b_0$ (if any).
\end{proof}

Note that instead of~\eqref{eqncorbm} it would also suffice to assume that for each fixed value $c\in(a,b_0)\cap(a,b_1)$,
\begin{align*}
 M_0(z) - M_1(z) = f(z) + \OO\left(\frac{1}{\phi_0(z,c)^2}\right),
\end{align*}
as $|z|\rightarrow\infty$ along our nonreal rays and that $M_1(z_0)= M_0(z_0)+f(z_0)$ for one fixed nonreal $z_0\in\C\backslash\R$.

\section{Uniqueness results for operators with discrete spectra}
\label{sec:urds}

Now we are finally able to investigate when the spectral measure determines the potential for operators with purely discrete spectrum.
In this respect, observe that the uniqueness results for the singular Weyl function from the previous sections do not  immediately yield
such results. In fact, if $\rho_0=\rho_1$, then the difference of the corresponding singular Weyl functions is an entire function by Theorem~\ref{IntR}.
However, in order to apply Corollary~\ref{corbm} we would need some bound on the growth order of this function. 
Fortunately, in the case of purely discrete spectrum with finite convergence exponent, a refinement of the arguments in the proof
of Theorem~\ref{thmbm} shows that the growth condition is not necessary.

\begin{corollary}\label{corbmdis}
Suppose $\phi_0(z,x)$, $\phi_1(z,x)$ are of growth order at most $s$ for some $s>0$ and  $\phi_1(z,x) \sim \phi_0(z,x)$ for an $x\in(a,b_0)\cap(a,b_1)$ as $|z|\to\infty$ along some nonreal rays dissecting the complex plane into sectors of opening angles less than $\nicefrac{\pi}{s}$.
Furthermore, assume that $H_0$ and $H_1$ have purely discrete spectrum with convergence exponent at most $s$.
If 
\be
 M_1(z) - M_0(z) = f(z), \quad z\in\C\backslash\R,
\ee
for some entire function $f(z)$, then $H_0=H_1$.
\end{corollary}

\begin{proof}
 It suffices to show that the functions $G_x$, $x\in(a,b_0)\cap(a,b_1)$, defined by~\eqref{eqnthmbmentfun} in Theorem~\ref{thmbm} satisfy a growth restriction as in~\eqref{eqnbmaltgrowth}.
 Because of Theorem~\ref{thm:IOphiev}, there are real entire solutions $\chi_0(z,x)$, $\chi_1(z,x)$ of growth order at most $s$, which are square integrable near the right endpoint and satisfy the boundary condition there if necessary.
 These solutions are related to the Weyl solutions $\psi_0(z,x)$, $\psi_1(z,x)$ by
 \begin{align*}
  \psi_j(z,x) = \frac{\chi_j(z,x)}{W(\chi_j,\phi_j)(z)}, \quad x\in(a,b_j),~ z\in\rho(H_j),~ j=0,1.
 \end{align*}
 Using the definition of the singular Weyl functions, we get for each $x\in(a,b_0)\cap(a,b_1)$, 
 \begin{align*}
G_x(z) & = \phi_1(z,x)  \theta_0(z,x) - \phi_0(z,x) \theta_1(z,x) - f(z)\phi_0(z,x) \phi_1(z,x) \\
 & = \phi_1(z,x) \psi_0(z,x) - \phi_0(z,x) \psi_1(z,x) \\
 & = \frac{\phi_1(z,x) \chi_0(z,x)}{W(\chi_0,\phi_0)(z)} - \frac{\phi_0(z,x) \chi_1(z,x)}{W(\chi_1,\phi_1)(z)}, \quad z\in\C\backslash\R.
\end{align*}
 Now since the numerators and denominators in the last line are of growth order at most $s$, Lemma~\ref{lemesthp} shows that there is some increasing sequence $r_n\uparrow\infty$ and positive constants $A_x, B_x$ such that 
\begin{align*}
 \sup_{|z|=r_n} |G_x(z)| \leq B_x \E^{A_x r_n^{\tilde{s}}}, \quad n\in\N,
\end{align*}
where $\tilde{s}>s$ such that our nonreal rays dissect the complex plane into sectors of opening angles less than $\nicefrac{\pi}{\tilde{s}}$.
Now use the remark after Theorem~\ref{thmbm}.
\end{proof}

Now the lack of a growth restriction in Corollary~\ref{corbmdis} implies that it immediately translates into a corresponding uniqueness result for the spectral measure.

\begin{theorem}\label{thmSpectFuncDisc}
Suppose that $\phi_0(z,x)$, $\phi_1(z,x)$ are of growth order at most $s$ for some $s>0$ and $\phi_1(z,x)\sim\phi_0(z,x)$ for an $x\in(a,b_0)\cap(a,b_1)$ as $|z|\rightarrow\infty$ along some nonreal rays dissecting the complex plane into sectors of opening angles less than $\nicefrac{\pi}{s}$.  
Furthermore, assume that $H_0$ and $H_1$ have purely discrete spectrum with convergence exponent at most $s$. If the corresponding spectral measures $\rho_0$ and $\rho_1$ are equal, then we have $H_0=H_1$.
\end{theorem}

\begin{proof}
Since the spectral measures are the same, Theorem~\ref{IntR} shows that the difference of the corresponding singular Weyl functions is an entire function and Corollary~\ref{corbmdis} is applicable.
\end{proof}

It should be emphasized that a similar result  has been proven in \cite{je} using the theory of de Branges spaces. However, the assumptions on $\phi_j(z,x)$ are of a different nature
and we are not aware of how to verify the assumptions on $\phi_j(z,x)$ to apply \cite[Theorem~4.1]{je} to the examples envisaged in the present paper. Nevertheless,
in some sense our assumptions here are stronger since they exclude (in the case $(a,b)=\R$) the possibility that one potential is a translation of the other
(which clearly would leave the spectral measure unchanged).

Note that in the case of discrete spectra, the spectral measure is uniquely determined by the eigenvalues $\lambda_n$ and the corresponding norming constants
\be
 \gamma_n^2 = \int_a^b \phi(\lambda_n,x)^2 dx,
\ee
since in this case we have
\be
 \rho = \sum_n \gamma_n^{-2} \delta_{\lambda_n},
\ee
where $\delta_{\lambda}$ is the unit Dirac measure in the point $\lambda$.

As another application, we are also able to prove a generalization of Hochstadt--Lieberman type uniqueness results.
To this end, let us consider an operator $H$ whose spectrum is purely discrete and has convergence exponent (at most) $s$.
Since the operator $H^D_c = H^D_{(a,c)} \oplus H^D_{(c,b)}$ with an additional Dirichlet boundary condition at $c$ is
a rank one perturbation of $H$, we conclude that the  convergence exponents of both $H^D_{(a,c)}$ and $H^D_{(c,b)}$ are
at most $s$ and hence by Theorem~\ref{thm:IOphiev} there are real entire solutions $\phi(z,x)$ and $\chi(z,x)$ of growth order at most $s$ which are
in the domain of $H$ near $a$ and $b$, respectively.

\begin{theorem}\label{thmHL}
Suppose $H_0$ is an operator with purely discrete spectrum of finite convergence exponent $s$. Let $\phi_0(z.x)$ and $\chi_0(z,x)$ be
entire solutions of growth order at most $s$ which lie in the domain of $H_0$ near $a$ and $b$, respectively, and suppose there is a $c\in (a,b)$ such that
\be\label{eqeahl}
\frac{\chi_0(z,c)}{\phi_0(z,c)}=\OO(1),
\ee
as $|z|\rightarrow\infty$ along some nonreal rays dissecting the complex plane into sectors of opening angles less than $\nicefrac{\pi}{s}$.
 Then every other isospectral operator $H_1$ for which $q_1(x)=q_0(x)$ almost everywhere on $(a,c)$ and which is associated with the same boundary
condition at $a$ (if any) is equal to $H_0$.
\end{theorem}

\begin{proof}
Start with some solutions $\phi_j(z,x)$, $\chi_j(z,x)$ of growth order at most $s$ and note that we can choose $\phi_1(z,x) = \phi_0(z,x)$ for $x\le c$ since $H_1$ and $H_0$
are associated with the same boundary condition at $a$ (if any). Moreover, note that we have $\phi_0(z,x) \sim \phi_1(z,x)$ as $|z|\to\infty$
along every nonreal ray even for fixed $x>c$ by Lemma~\ref{lemPhiAs}. Next note that the zeros of the Wronskian $W(\phi_j,\chi_j)$ are
precisely the eigenvalues of $H_j$ and thus, by assumption, are equal. Hence by the Hadamard factorization theorem $W(\phi_1,\chi_1)= \E^{g} W(\phi_0,\chi_0)$
for some polynomial $g$ of degree at most $s$. Since we can absorb this factor in $\chi_1(z,x)$, we can assume $g=0$ without loss of generality. Hence we have
\begin{align*}
 1= \frac{W(\phi_0,\chi_0)}{W(\phi_1,\chi_1)} & = \frac{\phi_0(z,x)\chi_0'(z,x) - \phi_0'(z,x)\chi_0(z,x)}{\phi_1(z,x)\chi_1'(z,x) - \phi_1'(z,x)\chi_1(z,x)} \\
      & = \frac{\phi_0(z,x)}{\phi_1(z,x)} \frac{\chi_0(z,x)}{\chi_1(z,x)} \left(\frac{\chi_0'(z,x)}{\chi_0(z,x)} - \frac{\phi_0'(z,x)}{\phi_0(z,x)}\right)
      \left(\frac{\chi_1'(z,x)}{\chi_1(z,x)} - \frac{\phi_1'(z,x)}{\phi_1(z,x)}\right)^{-1}
\end{align*}
and by virtue of the well-known asymptotics (see~\cite[Lemma~9.19]{tschroe}) 
\[
 \frac{\chi_j'(z,x)}{\chi_j(z,x)} = - \sqrt{-z} + \OO(1) \quad\text{and}\quad  \frac{\phi_j'(z,x)}{\phi_j(z,x)} = \sqrt{-z} + \OO(1), \quad j=0,1,
\]
as $|z|\rightarrow\infty$ along any nonreal rays, we conclude $\chi_1(z,x) \sim \chi_0(z,x)$ as well.

Furthermore, equality of the Wronskians implies
\[
\chi_1(z,x) = \chi_0(z,x) + F(z) \phi_0(z,x), \quad x \le c,
\]
for some entire function $F(z)$ of growth order at most $s$. Moreover, our assumption \eqref{eqeahl} implies that
\begin{align*}
 F(z) = \frac{\chi_1(z,c)-\chi_0(z,c)}{\phi_0(z,c)} = \frac{\chi_0(z,c)}{\phi_0(z,c)}\left(\frac{\chi_1(z,c)}{\chi_0(z,c)} -1\right)
\end{align*}
vanishes along our rays and thus it must be identically zero by the Phragm\'en--Lindel\"of theorem. Finally, choosing
$\theta_j(z,x)$ such that $\theta_1(z,x) = \theta_0(z,x)$ for $x\le c$ implies that the associated singular Weyl functions
are equal and the claim follows from Corollary~\ref{corbmdis}.
\end{proof}

Note that by \eqref{IOphias} the growth of $\phi_0(\,\cdot\,,c)$ will increase as $c$ increases while (by reflection) the growth of $\chi_0(\,\cdot\,,c)$
will decrease. In particular, if \eqref{eqeahl} holds for $c$, then it will hold for any $c'>c$ as well.

As a first example we give a generalization of the Hochstadt--Lieberman result from \cite{hl} to operators on $[0,1]$ with
Bessel-type singularities at both endpoints. Note that the case $k=l=0$ and $q_0\in L^1(0,1)$ is the classical Hochstadt--Lieberman result.

\begin{theorem}
Let $l,k\ge -\frac{1}{2}$.
Consider an operator of the form
\be
H_0 = -\frac{d^2}{dx^2} + q_0(x), \qquad q_0(x)= \frac{l(l+1)}{x^2} + \frac{k(k+1)}{(1-x)^2} +\ti{q}_0(x), \quad x\in(0,1),
\ee
with $\ti{q}_0$ satisfying
\be
f_l(x)f_k(1-x) \ti{q}_0(x) \in L^1(0,1), \qquad f_l(x)=\begin{cases} x, & l>-\nicefrac{1}{2},\\ (1-\log(x))x, & l= -\nicefrac{1}{2}.\end{cases}
\ee
If $l<\frac{1}{2}$, we choose the Friedrichs extension associated with the boundary condition
\be
\lim_{x\to0} x^l ( (l+1)\phi(x) - x \phi'(x))=0
\ee
at $0$ and similarly at $1$ if $k<\frac{1}{2}$.

Suppose $H_1$ satisfies $q_1(x)=q_0(x)$ for $x\in(0,\nicefrac{1}{2}+\eps)$ and has the same boundary condition at $0$ if $l <\frac{1}{2}$,
where $\eps=0$ if $k\ge l$ and $\eps>0$ if $k<l$. Then, $H_0=H_1$ if both have the same spectrum.
\end{theorem}

\begin{proof}
Immediate from Theorem~\ref{thmHL} together with the asymptotics of solutions of $H_0$ given in Lemma~2.2 (see in particular (2.24)) of \cite{kst}.
\end{proof}

In particular, this applies for example to the P\"oschl--Teller operator
\be
H= -\frac{d^2}{dx^2} + \frac{\nu(\nu+1)}{\sin(\pi x)^2}, \quad \nu\ge 0,
\ee
whose spectrum is given by $\sig(H)= \{ \pi^2(n+\nu)^2\}_{n\in\N}$ which plays an important role as an explicitly solvable model in physics.

As another result we can generalize Theorem~1.4 from \cite{gs2}.

\begin{corollary}[\cite{gs2}]
Let $H_0$ be an operator on $(-a,a)$ with purely discrete spectrum which is bounded from below and has convergence exponent $s<1$.
If $q_0(x)\ge q_0(-x)$ for $x> 0$, then $q_0$ on $(-a,0)$ and the spectrum uniquely determine $H_0$.
\end{corollary}

\begin{proof}
Since the convergence exponent satisfies $s<1$, our solutions are given by
\[
\phi_0(z,0)= \prod_{n\in\N} E_0(\mu_{-,n},z), \qquad \chi_0(z,0)= \prod_{n\in\N} E_0(\mu_{+,n},z),
\]
where $\mu_{\pm,n}$ are the Dirichlet eigenvalues on $(0,\pm a)$, respectively.
By our assumption the min-max principle implies $\mu_{-,n} \ge \mu_{+,n}$ and \eqref{eqeahl} follows from monotonicity of
$E_0(\mu,\I y)$ as $y\to\infty$.
\end{proof}

Note that, as pointed out in \cite[Proposition~5.1]{gs2}, the convergence exponent will satisfy $s\le 1-\frac{\eps}{4+2\eps}$
provided $q(x) \ge C |x|^{2+\eps} - D$ for some $C,D,\eps>0$ and our result is indeed a generalization of \cite[Theorem~1.4]{gs2}.
In fact, it was conjectured in \cite{gs2} and later proven in \cite{kh} that the restriction on the convergence exponent is
indeed superfluous. We will show how to replace the condition $q_0(x)\ge q_0(-x)$ by an asymptotic condition in the
next section.

\section{Perturbed harmonic oscillators}\label{secPQHO}
\label{sec:pho}

In this section we want to investigate perturbations of the quantum mechanical harmonic oscillator and in
particular its isospectral class originally considered by McKean and Trubowitz \cite{mktr} as well as Levitan \cite{le}
(see also Remark~4.8 in \cite{gst}).
We will build on some recent results from Chelkak, Kargaev and Korotyaev, \cite{chelkak}, \cite{ckk}, \cite{ckk2}.

Let $q$ be a real-valued function on $\R$ such that 
\be\label{condho}
  \int_\R \frac{|q(t)|}{1+|t|} dt <\infty
\ee
and consider the Schr\"{o}dinger operator
\be
 H = - \frac{d^2}{dx^2} + x^2 + q(x), \quad x\in\R.
\ee
We will see shortly that $H$ is bounded from below and hence in the l.p.\ case at both endpoints by a Povzner--Wienholtz type argument (cf.\ \cite[Lemma~C.1]{ge}).
Hence no boundary conditions are needed and the associated operator is unique. Moreover, the case $q\equiv 0$ is of course the famous
harmonic oscillator which can be solved explicitly in terms of Weber functions  $D_\nu$ (or parabolic cylinder functions) on $\R$. We refer the reader to
(e.g.) \S16.5 in \cite{whwa} for basic properties of these functions. In particular, two real entire solutions of the unperturbed equation
\be
- \phi_0''(z,x) + x^2 \phi_0(z,x) = z \phi_0(z,x), \quad x\in\R,~z\in\C
\ee
are given by
\be
\phi_{0,\pm}(z,x) = D_{\frac{z-1}{2}}(\pm\sqrt{2}x), \quad x\in\R,~z\in\C.
\ee
They are known to have the spatial asymptotics
 \begin{align}\label{asymD}
  \phi_{0,\pm}(z,x) & \sim (\pm\sqrt{2}x)^{\frac{z-1}{2}} \E^{-\frac{x^2}{2}}, \\
   \phi_{0,\pm}'(z,x) & \sim -2^{-\frac{1}{2}} (\pm\sqrt{2}x)^{\frac{z+1}{2}} \E^{-\frac{x^2}{2}}, 
\end{align}
as $x\rightarrow\pm\infty$, uniformly for all $z$ in bounded domains. In particular, this guarantees that $\phi_{0,\pm}(z,\cdot\,)$ are square integrable near $\pm\infty$.

By virtue of the usual perturbation techniques one can show that \eqref{condho} has solutions which asymptotically look like
the unperturbed ones. Details can be found in \cite[Section~2]{chelkak} or \cite[Section~3]{ckk}. We collect the relevant results plus some
necessary extensions in the following theorem.
 
\begin{theorem}\label{thmpertharmPHI}
There are unique solutions $\phi_\pm(z,x)$ of
\be
   -\phi_\pm''(z,x) + \left(x^2+q(x)\right) \phi_\pm(z,x) = z \phi_\pm(z,x), \quad x\in\R,~z\in\C,
\ee
such that for each $z\in\C$ we have the spatial asymptotics
\be
    \phi_\pm(z,x) \sim D_{\frac{z-1}{2}}(\pm\sqrt{2}x) \quad\text{and}\quad \phi_\pm'(z,x)\sim \pm\sqrt{2} D_{\frac{z-1}{2}}'(\pm\sqrt{2}x),
\ee
as $x\rightarrow \pm\infty$.
 Moreover, for each $x\in\R$ the functions $\phi_\pm(\,\cdot\,,x)$ and $\phi_\pm'(\,\cdot\,,x)$ are real entire functions of growth order at most one.
\end{theorem}
 
\begin{proof}
Existence and analyticity of these solutions is proved in~\cite[Section~2]{chelkak} respectively in~\cite[Section~3]{ckk}.
Uniqueness follows from the required asymptotic behavior and it remains to show that these solutions are of growth order at most one.
First of all note that 
\begin{align*}
  \phi_\pm(z,x) = \phi_\pm(z,0) c(z,x) + \phi_\pm'(z,0) s(z,x), \quad x\in\R,~z\in\C,
\end{align*}
where $s(z,x)$ and $c(z,x)$ are the solutions with the initial conditions
\begin{align*}
  s(z,0) = c'(z,0) = 0 \quad\text{and}\quad s'(z,0) = c(z,0) = 1, \quad z\in\C.
\end{align*}
Hence it suffices to show that the entire functions $\phi_\pm(\,\cdot\,,0)$ and $\phi_\pm'(\,\cdot\,,0)$ are of growth order at most one.    
Therefore we will need the estimate
\begin{align}\tag{$*$}\label{eqnSWTinequality}
  1 + \left|z\right|^{\frac{1}{12}} + \left|t^2-z\right|^{\frac{1}{4}} \geq \frac{\sqrt{|t|+1}}{\sqrt{|z|}}, \quad t\in\R,~z\in\C,~|z|\geq 2.
\end{align}
Note that by the reverse triangle inequality we only have to check this for all real $z\in[2,\infty)$. If 
\begin{align*}
 |t| \leq z (z-z^{-1})^{-\frac{1}{2}},
\end{align*}
then from $z\geq 2$ one sees that $|t|+1 \leq 2z$ and hence~\eqref{eqnSWTinequality} holds. 
Otherwise, one ends up with
\begin{align*}
   z |t^2-z|^{\frac{1}{2}} \geq |t|,
\end{align*} 
and again~\eqref{eqnSWTinequality} holds.
Using inequality~\eqref{eqnSWTinequality} we get the following bound for the functions
 \begin{align*}
   \beta_\pm(z) & = \pm \int_0^{\pm\infty} \frac{|q(t)|}{\left(1+|z|^{\frac{1}{12}}+|t^2 - z|^{\frac{1}{4}}\right)^2} dt  \\
    & \leq \pm |z| \int_0^{\pm\infty} \frac{|q(t)|}{|t|+1}dt, \quad z\in\C,~|z|\geq 2.
\end{align*}   
Now the estimates given in~\cite[Lemma~3.2]{ckk} or in~\cite[Corollary~2.6]{chelkak} show
\begin{align*}
    |\phi_\pm(z,0)|,\,|\phi_\pm'(z,0)| & \leq B \E^{A |z|\log|z|}, \quad z\in\C,~|z|\geq 2,    
\end{align*}
for some positive constants $A,B$, which proves the claim. 
\end{proof}
  
In particular, \eqref{asymD} shows that the solutions $\phi_\pm(z,x)$ always lie in $L^2(0,\pm\infty)$.
  
\begin{corollary}
The spectrum of $H$ is purely discrete, bounded from below, and has convergence exponent at most one.
\end{corollary}

\begin{proof}
   The existence of real entire solutions which lie in $L^2(\R)$ near $\pm\infty$ guarantees the discreteness of the spectrum.
   Since the eigenvalues of $H$ are the zeros of the entire function of growth order at most one
   \begin{align*}
     W(\phi_+,\phi_-)(z) = \phi_+(z,0) \phi_-'(z,0) - \phi_+'(z,0) \phi_-(z,0), \quad z\in\C,
   \end{align*}
   the spectrum has convergence exponent at most one. To see that the operator is bounded from below note that by the
   spatial asymptotics the underlying differential equation is non-oscillatory.
\end{proof}

In order to apply our uniqueness result Theorem~\ref{thmSpectFuncDisc}, we need high energy asymptotics of the solutions $\phi_\pm(z,x)$.

\begin{lemma}\label{lempertharmphiasym}
For each $x\in\R$ the unperturbed solutions have the asymptotics
\be
  D_{\frac{z-1}{2}}(\pm\sqrt{2}x) = 2^{\frac{z-1}{4}} \frac{\sqrt{\pi}}{\Gamma\left(\frac{3-z}{4}\right)} \E^{\pm x \sqrt{-z}} \left(1 + \OO\left(1/\sqrt{-z}\right)\right)
\ee
as $|z|\rightarrow\infty$ along each nonreal ray.
 Moreover, 
 for each $x\in\R$ the solutions $\phi_\pm(z,x)$ have the asymptotics
 \be
  \phi_\pm(z,x) \sim D_{\frac{z-1}{2}}(\pm\sqrt{2}x),
 \ee
 as $|z|\rightarrow\infty$ along each nonreal ray.
\end{lemma}

\begin{proof}
First of all note that because of the asymptotics in Lemma~\ref{lemPhiAs}, it suffices to consider the case $x=0$. Then the first claim is immediate
from
 \begin{align*}\tag{$*$}\label{eqnDz}
  D_{\frac{z-1}{2}}(0) = 2^{\frac{z-1}{4}} \frac{\sqrt{\pi}}{\Gamma\left(\frac{3-z}{4}\right)}, \quad z\in\C. 
 \end{align*}
For the second claim note that the estimates in~\cite[Lemma~3.2]{ckk} or in~\cite[Corollary~2.6]{chelkak} show that
  \begin{align*}
   \left| \phi_\pm(z,0) - D_{\frac{z-1}{2}}(0)\right| \leq C \Phi(z) \beta_\pm(z) \E^{A \beta_\pm(z)}, \quad z\in\C,
  \end{align*}
 for some constant $C$, $A\in\R$ and the function
 \begin{align*}
  \Phi(z) = \left|\frac{z}{2\E}\right|^{\frac{\re z}{4}} \E^{\frac{\pi-\alpha}{4}\im z} (1+|z|)^{-\frac{1}{4}}, \quad z=|z|e^{\I\alpha},~\alpha\in[0,2\pi).
 \end{align*} 
 Lebesgue's dominated convergence theorem shows that $\beta_\pm(z)$ converges to zero as $|z|\rightarrow\infty$ along each ray except the positive real axis.
 Using \eqref{eqnDz} and Stirling's formula for the Gamma function
 \begin{align*}
  \Gamma(z) \sim \E^{-z} z^z \sqrt{\frac{2\pi}{z}},
 \end{align*}
 as $|z|\rightarrow\infty$ along rays except the negative real axis,
 we get for each fixed $\alpha\in(0,2\pi)$
 \begin{align}\label{eqnpertharmfraction}
  \left|\frac{\Phi(z)}{D_{\frac{z-1}{2}}(0)}\right| & = \OO\left(\E^{\frac{\re z}{4}\left(\log|z|-\log|3-z|\right)} \E^{\frac{\im z}{4}\left(\pi-\alpha + \im\log\frac{3-z}{4}\right)} \right),
 \end{align}
  as $|z|\rightarrow\infty$ along the ray with angle $\alpha$.
  Now since
  \begin{align*}
   \log|3-z|-\log|z| =  \log\left| 1-\frac{3}{z}\right| = \OO\left(\frac{1}{z}\right),  
  \end{align*}
  as well as
  \begin{align*}
   \pi-\alpha + \im\log \left(\frac{3-z}{4}\right) = \im\left( \log(3-z) - \log(-z)\right) = \im \log\left( 1-\frac{3}{z}\right) = \OO\left(\frac{1}{z}\right),
  \end{align*}
  as $|z|\rightarrow\infty$ along each ray except the positive real axis, the fraction in~\eqref{eqnpertharmfraction} is bounded along each ray except the positive real axis, which proves the claim.
\end{proof}

Denote by $\lambda_n$, $n\in\N$, the eigenvalues of $H$ in increasing order. Associated to each eigenvalue is a left and right norming constant
\be
\gamma_{n,\pm}^2 = \int_\R \phi_\pm(\lambda_n,x)^2 dx, \quad n\in\N.
\ee
Now an application of Theorem~\ref{thmSpectFuncDisc} yields the following uniqueness theorem.

\begin{theorem}
Let $q$ be a real-valued function satisfying \eqref{condho}.
Then the eigenvalues together with the left or right norming constants determine the function $q$ uniquely.
\end{theorem}

\begin{proof}
 Let $q_0$, $q_1$ be two functions satisfying the condition \eqref{condho}.
 From Theorem~\ref{thmpertharmPHI} and Lemma~\ref{lempertharmphiasym} we infer that the corresponding solutions $\phi_{0,\pm}(z,x)$, $\phi_{1,\pm}(z,x)$ are of growth order at most one and have the same asymptotics on each nonreal ray. Since the spectra of these operators have convergence exponent at most one, the claim follows from Theorem~\ref{thmSpectFuncDisc}.
\end{proof}

This result is very close to \cite[Theorem~1.1]{ckk} with the main advantage that our condition \eqref{condho} is somewhat more explicit than
Condition A from \cite{ckk}, which reads
\be
\int_\R \frac{|q(x)|}{(1 + \left|z\right|^{\frac{1}{12}} + \left|t^2-z\right|^{\frac{1}{4}})^2} \le b(|z|) \|q\|_B
\ee
for some function $b$ which decreases to zero and some Banach space of real-valued functions $B$. Clearly, \eqref{condho} must hold
for the left-hand side to be finite.
Also note that their norming constants $\nu_n$, $n\in\N$, are related to ours via
\be
 \gamma_{\pm,n}^2 = (-1)^n \E^{\mp\nu_n} \dot{W}(\lambda_n), \quad n\in\N,
\ee
where $W$ is the Wronskian 
\be
 W(z)= \phi_-(z,x)\phi_+'(z,x) - \phi_-'(z,x)\phi_+(z,x), \quad x\in\R,~z\in\C,
\ee
and the dot indicates differentiation with respect to $z$.

An application of Theorem~\ref{thmHL} yields a Hochstadt--Lieberman type uniqueness result for the perturbed quantum harmonic oscillator extending Theorem 1.4 from \cite{gs2}.

\begin{theorem}
Let $q$ be a real-valued function satisfying \eqref{condho}.
Then $q$ on $\R_+$ or $\R_-$ together with the eigenvalues determine the function $q$ uniquely.
\end{theorem}

\begin{proof}
 By Lemma~\ref{lempertharmphiasym} we have $\phi_+(z,0)\sim\phi_-(z,0)$ as $|z|\rightarrow\infty$ along nonreal rays.
  Hence the claim immediately follows from Theorem~\ref{thmHL}.
\end{proof}

\bigskip
\noindent
{\bf Acknowledgments.}
We thank Rostyslav Hryniv for hints with respect to the literature.
G.T.\ gratefully acknowledges the stimulating
atmosphere at the Isaac Newton Institute for Mathematical Sciences in Cambridge
during October 2011 where parts of this paper were written as part of  the international
research program on Inverse Problems.

\end{document}